\documentclass[11pt]{amsart}
\usepackage[left=1in,right=1in,top=1in,bottom=1in]{geometry}

\usepackage{amsmath,amsthm,amssymb}
\usepackage{enumerate}
\usepackage[all,cmtip]{xy}

\theoremstyle{plain}
\newtheorem{theorem}{Theorem}[section]
\newtheorem{lemma}[theorem]{Lemma}
\newtheorem{corollary}[theorem]{Corollary}
\newtheorem{proposition}[theorem]{Proposition}
\newtheorem{fact}[theorem]{Fact}

\theoremstyle{definition}
\newtheorem{question}[theorem]{Question}

\newtheorem{definition}[theorem]{Definition}
\newtheorem{remark}[theorem]{Remark}
\newtheorem{example}[theorem]{Example}

\def\cf{\mathrm{cf}}

\begin{document}

\title[Continuous weak selections on products and filter spaces]{The existence of continuous weak selections and 
orderability-type 
properties in products and filter spaces}
\author[K. Motooka]{Koichi Motooka}
\address{Graduate School of Science and Engineering, 
	Ehime University,
	Matsuyama, 790-8577, Japan}
\email{k-motooka0426@outlook.jp}
\author[D. Shakhmatov]{Dmitri Shakhmatov}
\email{dmitri.shakhmatov@ehime-u.ac.jp}
\author[T. Yamauchi]{Takamitsu Yamauchi}
\email{yamauchi.takamitsu.ts@ehime-u.ac.jp}

\thanks{{\em Correspondent author\/}: Dmitri Shakhmatov, tel: (81) 89 927-9558, fax: (81) 89 927-9560}


\thanks{The second listed author was partially supported by the Grant-in-Aid for Scientific Research~(C) No.~26400091 by the Japan Society for the Promotion of Science (JSPS), and the third listed author was supported by JSPS KAKENHI Grant number 26800040}

\begin{abstract}
Orderability, weak orderability and the existence of continuous weak selections on spaces with 
a single
non-isolated point 
and their products are discussed.
We 
prove that a closed continuous image $X$ of a suborderable space must be hereditarily paracompact provided that its product $X\times Y$ with some non-discrete space $Y$ has a separately continuous weak selection.
\end{abstract}

\maketitle	

\section{Introduction}

Throughout this paper, {\em all spaces are assumed to be Hausdorff topological spaces\/}.

Let $X$ be a space and 
$
\mathcal{F}_{2}(X)=\{F\subset X:1\leq|F|\leq2\}
$,
where $|F|$ is the cardinality of $F$.
We always consider $\mathcal{F}_{2}(X)$
with the Vietoris topology
generated by the base consisting of all sets of the form
\begin{align*}
\langle\mathcal{V}\rangle=\{S\in \mathcal{F}_{2}(X):S\subset \bigcup\mathcal{V}\text{ and }S\cap V\neq\emptyset\text{ for each }V\in \mathcal{V}\},
\end{align*}
where $\mathcal{V}$ runs over all finite families of open subsets of $X$. (It suffices to assume that $|\mathcal{V}|\le 2$ here.)
A function $\sigma :\mathcal{F}_{2}(X) \rightarrow X$ is called a \textit{weak selection\/} on $X$ if $\sigma (F)\in F$ for every $F\in \mathcal{F}_{2}(X)$.
A weak selection on a space $X$ is {\em continuous\/} if it is continuous with respect to the Vietoris topology on its domain $\mathcal{F}_{2}(X)$ and the original topology on the range $X$.

A relation $\preceq$ on $X$ is:
\begin{itemize}
\item  {\em total\/} if
$x\preceq y$ or $y\preceq x$ for every $x, y \in X$,
\item {\em antisymmetric\/} if $x,y\in X$, $x\preceq y$ and $y\preceq x$
imply $x=y$,
\item
{\em transitive\/} if $x,y,z\in X$, $x\preceq y$ and $y\preceq z$ imply $x\preceq z$.
\end{itemize}

 A relation $\preceq$ on $X$ satisfying all three conditions above is called a {\em linear order\/} on $X$. 
A space $(X,\tau)$ is \textit{orderable} (respectively,  \textit{weakly orderable} \cite{mill:1981}) if $\tau_{\preceq}=\tau$ (respectively, $\tau_{\preceq}\subset \tau$)
for some linear order $\preceq$ on $X$. 
A space $X$ is \textit{suborderable} if it is a subspace of some orderable space.
Obviously, every orderable space is suborderable, and every suborderable space is weakly orderable.
The converse implications do not hold in general.

Every weak selection $\sigma :\mathcal{F}_{2}(X) \rightarrow X$
 determines 
 the relation $\preceq_{\sigma}$ on $X$ 
 defined by letting $x\preceq_{\sigma}y$ if and only if  $\sigma (\{x,y\})=x$
 for $\{x,y\}\in \mathcal{F}_{2}(X)$.
This relation 
is both total and antisymmetric, but it could fail to be transitive.
 A total and antisymmetric relation $\preceq$ is called a \emph{selection relation\/} because the conjunction of these two conditions is equivalent to the equality $\preceq =\preceq_\sigma$ for some (unique) weak selection $\sigma$ on $X$ \cite{gutev-nogura:2005}.

Let
$\preceq$ be a selection relation on $X$.
For $x,y\in X$, the notion $x\prec y$ means that $x\preceq y$ and $x\neq y$.
For each $x, y\in X$, we define
\begin{align*}
(\leftarrow,x)_{\preceq}=\{z\in X:z\prec_{\sigma}x\}
&\text{, }
(x,\rightarrow)_{\preceq}=\{z\in X:x\prec_{\sigma}z\}
\text{, }
\\
(\leftarrow,x]_{\preceq}=\{z\in X:z\preceq_{\sigma}x\}
&\text{, }
[x,\rightarrow)_{\preceq}=\{z\in X:x\preceq_{\sigma}z\}
\text{, }
\\
(x,y)_{\preceq}=(x,\rightarrow)_{\preceq} \cap (\leftarrow,y)_{\preceq}
&\text{, }
(x,y]_{\preceq}=(x,\rightarrow)_{\preceq} \cap (\leftarrow,y]_{\preceq}\text{, and so on.}
\end{align*}

For every selection relation $\preceq$ on $X$,  
the family $\{ (\leftarrow,x)_{\preceq}:x\in X\}\cup\{ (x,\rightarrow)_{\preceq} : x \in X \}$ of half-intervals generates a topology on $X$ which we shall denote by $\tau_\preceq$.
The
topology $\tau_{\preceq_{\sigma}}$ on $X$ 
is known as a \textit{selection topology\/} determined by $\sigma$ \cite{gutev-nogura:2001}. 

If $\tau$ is a topology on a set $X$, then a weak selection $\sigma$ on a topological space $(X,\tau)$ is called \textit{separately continuous\/} provided that $\tau_{\preceq_{\sigma}}\subset \tau$
\cite{gutev-nogura:2010}.
Every continuous weak selection is separately continuous, 
but
the converse does not hold in general (\cite[Example 1.21]{artico-marconi-pelant-rotter-tkachenko:2002}, \cite[Example 4.3]{gutev-nogura:2010}).
However, to our knowledge, it is unknown whether there is a space with a separately continuous weak selection which does not admit a continuous one.

It is well known that every weakly orderable space has a continuous weak selection \cite[Lemma 7.5.1]{micheal:1951}.
It was proved in \cite[Theorem 2.7]{hrusak-martinez-ruiz:2009} that there exists a space with a continuous weak selection which is not weakly orderable.
	
The above results can be summarized as follows:

\begin{equation}
\label{four:properties}
\arraycolsep = .5mm
\begin{array}{rl}
\text{orderable}\to\text{suborderable}
&\to \text{weakly orderable}\\ 
&\to \text{admits a continuous weak selection}\\
&\to \text{admits a separately continuous weak selection}
\end{array}
\end{equation}

In this paper, every filter $p$ is assumed to be non-trivial and free, that is, $\emptyset \notin p$ and $\bigcap p=\emptyset$.
Following \cite{garcia-miyazaki-nogura:2013}, for a filter $p$ on an infinite set $X$,
the space $X_p=X\cup \{p\}$,
where $X$ is discrete and the neighborhoods of $p$ are of the form $P\cup\{p\}$ for $P\in p$, is called a \textit{filter space}. Every filter space has only one non-isolated point $p$, and every (Hausdorff) space with 
a single
non-isolated point can be described as a filter space.

The reversibility of implications in \eqref{four:properties} for filter spaces was studied in \cite{garcia-miyazaki-nogura-tomita:2013}, see Remark \ref{lemma105:product-selection}~(i).

Recently, a great deal of attention in the literature was paid to the problem of 
the existence of (separately) continuous weak selections on a product of two spaces, as well as, on the square of a space itself (\cite{garcia-miyazaki-nogura:2013}, \cite{gutev:2016}).
 Let us highlight some relevant results.

\begin{theorem} 
\label{corollary101:product-selection}
If a product $X\times Y$ of a space $X$ with a non-discrete space $Y$ admits a separately continuous weak selection, then $\psi (X)\leq a(Y)$.
\end{theorem}

We refer the reader to Section \ref{section:notation} for the definitions of cardinal functions $\psi (X)$ and $a(Y)$.
 
This theorem was proved first  in the case of continuous weak selections for products of filter spaces by Garc\'{i}a-Ferreira, Miyazaki and Nogura \cite[Theorem 2.2]{garcia-miyazaki-nogura:2013}. (It should be noted that the result for filter spaces 
implies the result for general spaces here.)
The generalization to the case of {\em separately\/} continuous weak selections can be derived from \cite[Proposition 3.4 and Remark 3.5]{gutev:2016}.
A slightly more general result can be found 
in \cite{motooka:2016}.

\begin{theorem}
\label{tot:disc}
If a product $X\times Y$ of a space $X$ with a non-discrete space $Y$ admits a separately continuous weak selection, then $X$ is totally disconnected.
\end{theorem}

This theorem was proved in  the case of continuous weak selections by Garc\'{i}a-Ferreira, Miyazaki and Nogura \cite[Theorem 3.4]{garcia-miyazaki-nogura:2013}. It was extended to the case of separately continuous weak selections by the first author in \cite{motooka:2016}.

The following theorem was proved 
in \cite{motooka:2016}
as an extension of \cite[Theorem 2.2]{garcia-miyazaki-nogura:2013}.

\begin{theorem}
\label{theorem1:product-selection}
Let $X$ be a non-discrete space and $\kappa$ a regular cardinal such that $a(X)<\kappa$.
Then $X\times S$ does not admit a separately continuous weak selection for every stationary subset $S$ of $\kappa$.
\end{theorem}

The following theorem was established recently by Gutev \cite[Corollary 5.3]{gutev:2016}.
\begin{theorem}
\label{thm:Gutev}
The square $X\times X$ of a regular countably compact space $X$ has a continuous weak selection if and only if $X$ is zero-dimensional and metrizable. 
\end{theorem}

Countable compactness of $X$ cannot be dropped in this theorem; see Remark \ref{remark:c-cpt-in-Gutev's-thm}.

The purpose of this paper is to study 
the existence of continuous weak selections and related properties in \eqref{four:properties} for filter spaces and 
product spaces.

In Sections \ref{section:Weak orderablity and suborderability of filter spaces} and \ref{uniform:filter}, we discuss properties in \eqref{four:properties} for filter spaces.
Some counterexamples of filter spaces will be given in Section \ref{section:examples}.
In Section \ref{section:produts-of-filter-spaces}, we give sufficient conditions for the weak orderability of products of filter spaces.
In Section \ref{section:necessary-conditions}, we prove a necessary condition for the existence of separately continuous weak selection on product spaces
which generalizes \cite[Theorem 3.1]{garcia-miyazaki-nogura:2013}; see Theorem \ref{theorem:necessary-cond-hered-para}.
Some open questions are listed in Section \ref{section:questions}.

\section{Notations}
\label{section:notation}

A subset $C$ of a linearly ordered set $(X,\le)$ is said to be {\em cofinal\/} in $(X,\le)$ provided that for every $x\in X$, there exists $c\in C$ such that $x\le c$. 

Every ordinal $\alpha$ is considered as a linearly ordered topological space consisting of ordinals less than $\alpha$. 
A cardinal is
an ordinal that cannot be mapped onto a smaller 
ordinal
bijectively.
For a cardinal $\kappa$, the least cardinal bigger than $\kappa$ is denoted by $\kappa^+$.
The first infinite ordinal (cardinal) and the first uncountable one are 
denoted by $\omega$ and $\omega_{1}$, respectively. 
Recall that the \textit{cofinality} $\cf(\kappa)$ of a cardinal $\kappa$ is the smallest cardinal $\sigma$ such that there exists a map
$f:\sigma\to\kappa$ such that the set $\{f(\alpha):\alpha<\sigma\}$ is cofinal in $\kappa$.
A cardinal $\kappa$ is {\em regular\/} if $\cf(\kappa)=\kappa$.
The cardinality of a set $A$ is denoted by $|A|$.

For a subset $A$ of a space $X$, we use $\overline{A}$ to denote the closure of $X$ in $X$.
For a space $X$ and a non-isolated point $p \in X$, the cardinal 
\begin{align*}
	a(p,X)=\min\{\kappa : p\in \overline{A}\text{ for some }A\subset X\setminus \{p\}\text{ with }|A|\le\kappa\}
\end{align*} 
is called the \textit{approaching number} of $p$ in $X$ \cite{garcia-miyazaki-nogura:2013}. It was introduced in \cite{garcia-gutev-nogura-sanchis-tomita:2002} under the name of selection approaching number.

For a space $X$ and $p \in X$, let $t(p,X)$ denote the {\em tightness of $p$ in $X$}, that is, the smallest infinite cardinal number $\kappa$ with the property that if $p \in \overline{A}$ where $A \subset X$, then there exists $B\subset A$ such that $p \in \overline{B}$ and $|B|\leq \kappa$.
Clearly, 
\begin{equation}
\label{a:and:t}
a(p,X)\le t(p,X)
\text{ for every non-isolated point }p\in X
\end{equation}
  but the equality does not hold in general; see \cite[Section 4]{garcia-gutev-nogura-sanchis-tomita:2002}.

It is convenient to introduce the cardinal 
\begin{align*}
a(X) =\min\{a(p,X) : p\text{ is a non-isolated point of }X\},
\end{align*}
which we call the {\em approaching number of $X$\/}.

It easily follows from \eqref{a:and:t} that
\begin{equation}
\label{a:t}
a(X)\le t(X),
\end{equation}
  where 
$t(X)=\sup\{t(p,X): p\in X\}$ is the {\em tightness\/} of $X$.

A point $p$ in a space $X$ is said to be a \textit{$G_{\kappa}$-point} if it is the intersection of $\kappa$-many open sets in $X$, and the cardinal 
\begin{align*}
	\psi(p,X)&=\min\{\kappa : p\text{ is a }G_{\kappa}\text{-point in }X\}
\end{align*}
is called the \textit{pseudo-character} of $p$ in $X$.
The \textit{pseudo-character} $\psi (X)$ of $X$ 
is defined by
\begin{align*}
\psi(X) &=\sup\{\psi(p,X) : p\in X\}. 
\end{align*}

By $\mathbb{Z}$ we denote the set of integers with the usual order.

For undefined notations and terminology, we refer the reader to \cite{engelking:1989} or \cite{kunen:1983}.

\section{Various forms of orderability for filter spaces}
\label{section:Weak orderablity and suborderability of filter spaces}

In this section, we consider properties in \eqref{four:properties} for filter spaces. We start with 
two comments
on the reversibility of some implications in \eqref{four:properties}.

\begin{remark}
	\label{lemma105:product-selection}
\begin{enumerate}[{\rm (i)}]
\item 
It was proved in \cite[Theorem 3.7]{garcia-miyazaki-nogura-tomita:2013} that {\em every suborderable filter space is orderable\/}.

\item
Applying \cite[Theorem 2.6 and Proposition 4.1]{gutev-nogura:2010} and the fact that every point in $X_p \setminus \{p\}$ is isolated, we  see that {\em every 
separately continuous weak selection on a filter space $X_p$ is continuous\/}.
\end{enumerate}
\end{remark}

In view of the above remark, for filter spaces, the implications in \eqref{four:properties} are simplified as follows:
\begin{equation}
\label{four:properties:for:filter:spaces}
\text{orderable}\to\text{weakly orderable}
\to\text{admits a continuous weak selection}.
\end{equation}

We also consider the following notions.
\begin{definition}[\cite{garcia-gutev-nogura-sanchis-tomita:2002}]
For a space $X$ and $p \in X$, $X$ is  said to be \emph{(weakly) $p$-orderable}
if it is (weakly) orderable 
by some linear order $\preceq$ such that $p$ is $\preceq$-maximal, or equivalently, $\preceq$-minimal.
\end{definition}

A family $\mathcal{P}$ of subsets of $X$ is said to be \emph{nested} if  $P \subset Q$ or $Q \subset P$ for every $P, Q \in \mathcal{P}$.

\begin{theorem}[{\cite[Corollary 5.5]{garcia-gutev-nogura-sanchis-tomita:2002}}]
\label{lemma:filter-sp-suff-cond-wo}
A filter space $X_p$ is weakly $p$-orderable  if and only if 
there exists a nested subfamily  $\mathcal{P}$ of $p$ satisfying $\bigcap\mathcal{P} = \emptyset$.
\end{theorem}

This theorem was obtained as a corollary of {\cite[Theorem 5.1]{garcia-gutev-nogura-sanchis-tomita:2002}}; see Remark \ref{remark:filter-sp-ord}  below. 
We will give another proof of Theorem \ref{lemma:filter-sp-suff-cond-wo}
in the end of Section \ref{section:produts-of-filter-spaces}.

\begin{lemma}
\label{element:in:the:intersection}
Let $\mathcal{P}$ be a nested family of subsets of a set $X$ such that $\bigcap\mathcal{P} =\emptyset$.
If $\mathcal{P}'\subset \mathcal{P}$ and 
$\bigcap \mathcal{P}'\neq\emptyset$, then 
$P \subsetneq \bigcap\mathcal{P}'$
for some $P\in\mathcal{P}$.
\end{lemma}
\begin{proof}
Let $\mathcal{P}'\subset \mathcal{P}$ with  
$\bigcap \mathcal{P}'\neq\emptyset$ and fix 
$x \in \bigcap \mathcal{P}'$.
Since $\bigcap\mathcal{P} =\emptyset$, we have $x \not\in  P$ for some $P\in\mathcal{P}$. 
If $P'\in\mathcal{P}'$, then $x\in P'\setminus P$, so $P'\not\subset P$. Since $P,P'\in\mathcal{P}$ and $\mathcal{P}$ is nested, $P\subset P'$. Since this holds for every $P'\in\mathcal{P}'$, we get $P\subset \bigcap\mathcal{P}'$.
Finally, $P\neq \bigcap\mathcal{P}'$, as 
$x\in (\bigcap\mathcal{P}')\setminus P$.
\end{proof}

\begin{proposition}
\label{proposition:p-ord->psi<a}
Every weakly $p$-orderable filter space $X_p$
satisfies
$\psi(X_p) \leq \kappa \le a(X_p)$,
where $$\kappa =\min\{|\mathcal{P}| : \mathcal{P} \text{ is a nested subfamily of } p \text{ satisfying } \bigcap\mathcal{P} =\emptyset \}.$$
(The cardinal $\kappa$ is well-defined by Theorem \ref{lemma:filter-sp-suff-cond-wo}.)
\end{proposition}

\begin{proof}
The inequality $\psi(X_p)\leq \kappa$ is immediate from the definitions of  $X_p$, $\psi(X_p)$ and $\kappa$.

To prove the inequality $a(X_p)\geq \kappa$, it suffices to fix  $A\subset X$ with $|A|<\kappa$ and show that $p\not\in\overline{A}$.
Since $\bigcap\mathcal{P} = \emptyset$,
for
each $a \in A$, there exists $P_a \in \mathcal{P}$ such that 
$a \notin P_a$.
Note that the family $\mathcal{P}'=\{P_a:a\in A\}\subset\mathcal{P}$ is nested, as a subfamily of the nested family $\mathcal{P}$. Since $|\mathcal{P}'|\le |A|<\kappa$, 
from the minimality of $\kappa$ we conclude that 
$\bigcap \mathcal{P}' \ne \emptyset$.
Applying Lemma \ref{element:in:the:intersection}, we can find 
$P\in\mathcal{P}$ such that $P\subset \bigcap\mathcal{P}'$.
Since $a\not\in P_a$ for every $a\in A$, we have $A\cap \bigcap\mathcal{P}'=A\cap \bigcap_{a\in A}{P_a}=\emptyset$.
This
shows that $A \cap P = \emptyset$, and hence $p \notin \overline{A}$.
\end{proof}

\begin{remark}
The converse of Proposition \ref{proposition:p-ord->psi<a} does not hold in general. Indeed, there exists a filter space $Z_p$ such that $\psi(Z_p)\leq a(Z_p)$, yet  $Z_p$ does not admit a continuous weak selection; see Example \ref{exotic}.
\end{remark}

For a filter $p$ on a set $X$, let $\| p\| $ denote the cardinal $\min \{ |P| :P \in p\} $ following \cite[p.144]{comfort-negrepontis:1974}.
A straightforward proof of the next proposition is omitted. 
\begin{proposition}
\label{basic:proposition:about:|p|}
$\max\{ a(X_p), \psi(X_p) \} \leq \| p\|  \leq |X_p|$
for every filter $p$ on a set $X$.
\end{proposition}

\begin{proposition}
\label{corollary:dispersion-character}
Let $X_p$ be a filter space such that $a(X_p)=\| p\|=\kappa $.
Then there exists a nested subfamily  $\{P_\alpha : \alpha < \kappa \} \subset p$ satisfying $P_\alpha \subsetneq \bigcap_{\beta<\alpha} P_\beta$ for every  $\alpha < \kappa$ and $\bigcap_{\alpha < \kappa} P_\alpha = \emptyset$.
In particular, $X_p$ is weakly $p$-orderable by Theorem \ref{lemma:filter-sp-suff-cond-wo}. 
\end{proposition}

\begin{proof}
Let 
$P=\{x_\alpha : \alpha < \kappa  \}$
be a faithfully enumerated element of $p$ witnessing the equality 
$\| p\|=\kappa $.
For every $\alpha < \kappa$, let $P_\alpha = \{ x_\beta :  \alpha \leq \beta < \kappa \}$.
Then $P_\alpha \in p$ for $\alpha < \kappa$.
Indeed, since $\alpha < \kappa = a(X_p)$, we have 
$p \notin \overline{\{ x_\gamma : \gamma  < \alpha \}}$ in $X_p$, and  there exists $Q \in p$ with 
$\{ x_\gamma : \gamma < \alpha\} \cap Q =\emptyset$.
Then $Q\cap P \subset P_\alpha $ and hence $P_\alpha \in p$. 
Thus $\{P_\alpha:\alpha<\kappa \}$ is the required nested subfamily of $p$.
\end{proof}

\begin{remark}
\label{remark:converse-a=delta->nested}
There exists a weakly $p$-orderable filter space $Z_p$ such that $a(Z_p) < \|p\|$;
see Example \ref{example:psi<a<card}.
\end{remark}

Proposition  \ref{corollary:dispersion-character} is applicable to 
all ultrafilters.

\begin{proposition}
\label{proposition:ultrafilter-dispersion-ch}
For every ultrafilter $p$ on a set $X$, 
we have $a(X_p) =\| p\| $.
\end{proposition}
 
\begin{proof}
Let $p$ be an ultrafilter on a set $X$.
Since $a(X_p) \leq \| p\| $ 
by Proposition \ref{basic:proposition:about:|p|},
it suffices to show $a(X_p) \geq \| p\| $.
Let $A$ be a set with $|A|<\| p\| $.
Then we have $A \notin p$. Since $p$ is an ultrafilter, $X\setminus A \in p$. This shows that $p \notin \overline{A}$ in $X_p$. Thus we have $a(X_p) \geq \| p\| $.
\end{proof}

From Propositions \ref{corollary:dispersion-character} and \ref{proposition:ultrafilter-dispersion-ch}, we get the following

\begin{corollary}[{\cite[Corollary 3.4]{garcia-miyazaki-nogura-tomita:2013}}]
\label{ultrafilter:is:weakly:orderable}
For every ultrafilter $p$ on a set $X$, the filter space $X_p$ is weakly orderable.
\end{corollary}

This corollary also follows 
from the next result of independent interest.

\begin{proposition}
\label{proposition:ultra-weak-ord}
If $p$ is an ultrafilter on $X$,
then every linear order on $X$ can be extended to a linear order $\preceq$ on $X_p$ such that $\tau_\preceq$ is coarser than the topology of $X_p$.
\end{proposition}

\begin{proof}
Let $\leq$ be a linear order on $X$.
Define a linear order $\preceq$ on $X_p$ as follows.
For $x,y\in X$, let $x\preceq y$ if and only if $x\le y$.
For each $x\in X$, 
since $p$ is a free ultrafilter on $X$,
either $(\leftarrow,x)_\leq \in p$ or $(x,\rightarrow)_\leq \in p$;
we define $p\preceq x$ in the former case and $x\preceq p$ in the latter case.
Finally, we also let $p\preceq p$.
By our definition, both intervals
$(x,\rightarrow)_\preceq$ and $(\leftarrow,x)_\preceq$ are open in $X_p$ for every $x \in X_p$.
Thus, the topology $\tau_\preceq$ generated by $\preceq$ is coarser than that of $X_p$.	\end{proof}

The following proposition follows from \cite[Corollary 3.4 and Theorem 3.7]{garcia-miyazaki-nogura-tomita:2013}. We give a direct proof here for the sake of completeness.

\begin{proposition}
\label{proposition:ultra-not-subord}
If $p$ is an ultrafilter on $X$, then
$X_p$ is not (sub)orderable.
\end{proposition}

\begin{proof}

Suppose for contradiction that $p$ is an ultrafilter and $X_p$ is suborderable.
Take a linearly ordered topological space
$(Y,\preceq)$ such that $X_p$ is a subspace of $Y$.
Since $p$ is a free ultrafilter, either
$(\leftarrow,p)_\preceq\cap X_p \in p$ or 
$(p,\rightarrow)_\preceq\cap X_p \in p$.
Without loss of generality, we may assume 
$(\leftarrow,p)_\preceq\cap X_p \in p$.
Then $(\leftarrow,p]_\preceq\cap X_p$ is open in $X_p$.

Note that  $(x,p)_\preceq \cap X_p \ne \emptyset$ for every $x \in (\leftarrow,p)_\preceq$ since $(x,\rightarrow)_\preceq \cap X_p $ is a neighborhood of $p$ and $p$ is a non-trivial filter.

We claim that there exist an ordinal $\alpha$ and disjoint cofinal subsets
$\{c_\beta : \beta<\alpha \}$ and  $\{d_\beta : \beta<\alpha \}$ of $I=(\leftarrow,p)_\preceq\cap X_p$ such that $c_\beta \prec d_\beta \prec c_\gamma$ if $\beta <\gamma <\alpha$.
Indeed, take $c_0, d_0 \in I$ with $c_0 \prec d_0$, 
and assume that $c_\beta$ and $d_\beta$ has been taken for 
$\beta<\gamma$.
If $I  \subset \bigcup _{\beta <\gamma } (\leftarrow, d_\beta]_\preceq$, then $\{c_\beta : \beta<\gamma \}$ and  $\{d_\beta : \beta<\gamma \}$ are the required cofinal subsets of $I$ (if one takes $\alpha=\gamma$).
If $I \not \subset \bigcup _{\beta <\gamma } (\leftarrow, d_\beta]_\preceq$, then we can select  $c_\gamma , d_\gamma \in I \setminus  \bigcup _{\beta <\gamma } (\leftarrow, d_\beta]_\preceq$ so that $c_\gamma \prec d_\gamma$. 
According to this construction,
we have  $I  \subset \bigcup _{\beta <\alpha } (\leftarrow, d_\beta]_\preceq$ for some ordinal $\alpha < |I|^+$. Then  $\alpha$, 
$\{c_\beta : \beta<\alpha \}$ and  $\{d_\beta : \beta<\alpha \}$ are as required.

Let $C=\{c_\beta : \beta < \alpha\}$.
Then both $C$ and $I \setminus C$ are cofinal sets in $(I,\preceq)$.
Since $p$ is an ultrafilter,
either $C \in p$ or $I \setminus C  \in p$. 
If $C \in p$, then there exists $x \in Y$ such that $(x,p]_\preceq \cap X_p \subset C\cup \{p\}$, which contradicts the fact that 
$I \setminus C$ is a cofinal set in $I=(\leftarrow,p)_\preceq\cap X_p$.
Similarly, we also have a contradiction in the case when
 $I \setminus C  \in p$.
\end{proof}

\begin{example}
There exists a  weakly $p$-orderable filter space $X_p$ which is not orderable.
Indeed, let $X=\omega$ with the usual order $\leq$ and $p$ be an ultrafilter containing  the Fr\'echet filter
$\{A\subset X:|X\setminus A|<\omega \}$.
Then the order $\leq$ can be extended to
a linear order $\preceq$ on $X_p$ as in Proposition \ref{proposition:ultra-weak-ord}, so $X_p$ is weakly $p$-orderable by $\preceq$. On the other hand, $X_p$ is not orderable by Proposition \ref{proposition:ultra-not-subord}.
\end{example}

\begin{remark}
\label{remark:ord-weak-p-ord}
There exists an orderable filter space $Z_p$ which is not weakly $p$-orderable.
Indeed,  there exists an orderable filter space $Z_p$ such that $a(Z_p) =\omega < \omega_1 =\psi(Z_p)$; see Example \ref{example:a<psi<card}.
Then $Z_p$ is not weakly $p$-orderable by Proposition \ref{proposition:p-ord->psi<a}.
\end{remark}

The following diagram summarizes main results in this section
that hold for all filter spaces.

\begin{center} \medskip\hspace{1em}\xymatrix{
&  \text{$p$-orderable}\ar[r]\ar[d]&\text{orderable}\ar[d]\\
a=\|p\|\ar[r]&	\text{weakly $p$-orderable}\ar[r]\ar[d]&\text{weakly orderable}\ar[d]\\
& \psi\le a  & \text{admits a continuous weak selection}
  }
\\
\bigskip
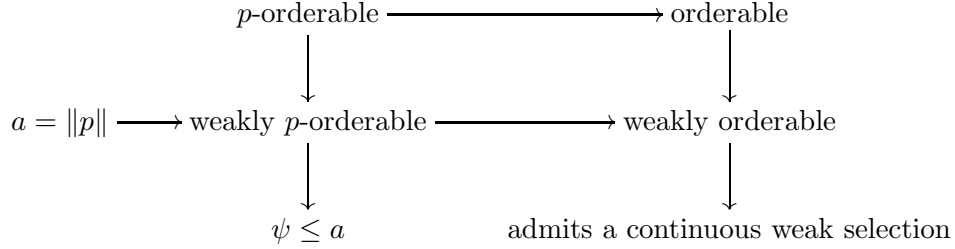
Diagram 1. Implications that hold for filter spaces.
\end{center}

\section{The uniform filter $p^\kappa(X)$}
\label{uniform:filter}

For a set $X$ and an infinite cardinal $\kappa$ with $\kappa \leq |X|$, we follow  \cite[p.144]{comfort-negrepontis:1974} to denote by $p^\kappa(X)$ the filter $\{A\subset X:|X\setminus A|<\kappa \}$ on $X$.
For simplicity, we use $p^\kappa$ instead of $p^\kappa (X)$ if there is no confusion. 

The next proposition computes cardinal invariants $a$ and $\psi$ of the filter space $X_{p^\kappa}$ showing that they are independent of each other.

\begin{proposition}
\label{a:psi:in:p(kappa)}
Let $X$ be a set and $\kappa$  an infinite cardinal satisfying $\kappa\le |X|$.
\begin{enumerate}[{\rm (i)}]
\item \label{a:psi:in:p(kappa)-1}$a(X_{p^\kappa})=\kappa$ and $\|p^\kappa\|=|X|$. 
\item \label{a:psi:in:p(kappa)-2}
If $|X|=\kappa$, then $\psi(X_{p^\kappa})=\cf(\kappa)$.
\item \label{a:psi:in:p(kappa)-3}
If $|X|>\kappa$, then 
$\psi(X_{p^\kappa})=|X|$.
\end{enumerate}
\end{proposition}

\begin{proof}
(\ref{a:psi:in:p(kappa)-1})
The fact that $\|p^\kappa\|=|X|$ follows from $|X|\ge\kappa$.
Let
$A\subset X$ and $|A|=\kappa$. Then
$A \cap P \ne \emptyset$ for every $P \in p^\kappa$,
which shows that $p^\kappa \in \overline{A}$,
and hence $a(X_{p^\kappa}) \leq \kappa$. 	
On the other hand, for every  $B\subset X$ with $|B|<\kappa$,
we have $X\setminus B\in p^\kappa$ and $B\cap (X\setminus B)=\emptyset$, which shows $p^\kappa\notin \overline{B}$.
Thus, $a(X_{p^\kappa}) \geq \kappa$.

(\ref{a:psi:in:p(kappa)-2}) 
It is well known that
$\cf (\kappa) = \min \{|\mathcal{A}| : \ \bigcup \mathcal{A} =\kappa$ and $|A|<\kappa$ for all $A\in\mathcal{A}\}$;
see, for example, \cite[Lemma 3.6]{jech:1997}.
This and $|X|= \kappa$ imply 
\begin{align*}
\psi(X_{p^\kappa})&=\psi(p^\kappa,X_{p^\kappa})=
\min\{|\mathcal{U}| : \mathcal{U}\subset p^\kappa \text{ and } \bigcap\mathcal{U}=\emptyset\}\\
&=
 \min\{|\mathcal{A}|: \bigcup\mathcal{A}=X
\text{ and } |A|<\kappa
\text{ for all } A\in\mathcal{A}
\}= \cf(\kappa).
\end{align*}

(\ref{a:psi:in:p(kappa)-3})
Assume $|X|>\kappa$ and let 
$\lambda$ be a cardinal satisfying $\kappa\le\lambda<|X|$.
For every $\mathcal{U} \subset p^\kappa$
with $\bigcap\mathcal{U}=\emptyset$,
we have  $\lambda^+\leq|X|=|\bigcup_{U\in \mathcal{U}}(X\setminus U)|$.
Since $\lambda^+$ is a regular cardinal and 
$|X\setminus U| <\kappa\leq \lambda $ for every $U \in \mathcal{U}$,
we have $|\mathcal{U}|\geq \lambda^+$.
Hence,
$\psi(X_{p^\kappa}) \geq  \sup \{\lambda ^+ : \kappa\le \lambda < |X| \} =|X|$.
This and Proposition \ref{basic:proposition:about:|p|} imply $\psi(X_{p^\kappa})=|X|$.
\end{proof}

\begin{proposition}
\label{theorem102:product-selection}
For an infinite cardinal $\kappa$ and 
a set $X$ satisfying  $\kappa\leq |X|$, 
the following conditions are equivalent:
\begin{enumerate}[{\rm (a)}]
\item $X_{p^\kappa}$ is weakly $p^\kappa$-orderable;
\item
$X_{p^\kappa}$ admits a (separately) continuous weak selection;
\item
$|X|=\kappa$.
\end{enumerate}
\end{proposition}
\begin{proof}
The implication (a) $\Rightarrow$ (b) follows from \eqref{four:properties}, as weakly $p^\kappa$-orderable spaces are weakly orderable.
The implication (b) $\Rightarrow$ (c) 
can be proved by the same argument as in \cite[Proposition 3]{douwen:1990}.
To establish the implication (c) $\Rightarrow$ (a), assume that
$|X|=\kappa$. Then $a(X_{p^\kappa})=\|p^\kappa\|=\kappa$ by  Proposition \ref{a:psi:in:p(kappa)} (\ref{a:psi:in:p(kappa)-1}),
so (a) follows from Proposition \ref{corollary:dispersion-character}.
\end{proof}

\begin{proposition}
\label{proposition:filter-sp-regular-orderable}
For an infinite cardinal $\kappa$ and 
a set $X$ satisfying  $\kappa\leq |X|$, 
the following conditions are equivalent:
\begin{enumerate}[{\rm (a)}]
\item
$X_{p^\kappa}$ is $p^\kappa$-orderable;
\item $X_{p^\kappa}$ is orderable;
\item
$|X|=\kappa$ and $\kappa$
is regular.
\end{enumerate}
\end{proposition}

\begin{proof}
(a) $\Rightarrow$ (b) is obvious.
To establish
(b) $\Rightarrow$ (c), 
assume that
$X_{p^\kappa}$ is orderable.
Since orderable spaces have a continuous weak selection,
Proposition \ref{theorem102:product-selection}
implies that $|X|=\kappa$.
Since $X_{p^\kappa}$ is orderable,
$\psi (X_{p^\kappa})=t(X_{p^\kappa})$; see \cite[3.12.4 (d)]{engelking:1989}. Furthermore,
$a(X_{p^\kappa})\le t(X_{p^\kappa})$ holds by \eqref{a:t}, which 
gives
$a(X_{p^\kappa})\le \psi (X_{p^\kappa})$. 
Since $\psi(X_{p^\kappa})\le |X|=\kappa$ by Proposition 
\ref{basic:proposition:about:|p|} and 
$a(X_{p^\kappa})=\kappa$ by Proposition \ref{a:psi:in:p(kappa)} (\ref{a:psi:in:p(kappa)-1}), we conclude that $\psi(X_{p^\kappa})=\kappa$.
On the other hard, $\psi(X_{p^\kappa})=\cf(\kappa)$ by 
Proposition \ref{a:psi:in:p(kappa)} 
(\ref{a:psi:in:p(kappa)-2}). This  shows that $\kappa$ is regular.

To show (c) $\Rightarrow$ (a), assume that $|X|=\kappa$ and $\kappa$ is regular.
If $\kappa =|X|= \omega$, then $X_{p^\kappa}$ is homeomorphic to the ordinal space $\omega+1$ by a homeomorphism mapping $p^\kappa$ to $\omega$, an hence $X_{p^\kappa}$ is $p^\kappa$-orderable.
From now on, we assume that $\kappa =|X|> \omega$.

Let $(Y,\leq)$ be the ordered subset $(\kappa \times \mathbb{Z})\cup \{(\kappa, 0)\}$ of the ordered set $(\kappa+1) \times \mathbb{Z}$ with the lexicographical order.
Let $Y$ be equipped with the order topology induced by $\leq$.
Then  $\{(y,\rightarrow)_\leq : y \in  Y\setminus \{ (\kappa,0)\}\}$ is a neighborhood base of the point $(\kappa, 0)$, and 
 every $y\in Y\setminus \{ (\kappa,0)\}$ is an isolated point of $Y$ satisfying $|Y \setminus (y,\rightarrow)|<\kappa$. 
 
 Since $|Y|=\kappa =|X_{p^\kappa}|$, we can take a bijection $f: X_{p^\kappa} \to Y$ such that 
 $f(p^\kappa) =(\kappa, 0)$. 
 Since $|Y \setminus (y,\rightarrow)|<\kappa$ for every $y\in Y\setminus \{ (\kappa,0)\}$, $f$ is continuous. 
 Since $\kappa$ is regular, for every $A\subset X$ with $|A|<\kappa$ there exists $y\in Y\setminus \{ (\kappa,0)\}$ 
 such that $f(a)<y$ for each $a \in A$. This shows that  $f$ is a homeomorphism. Hence $X_{p^\kappa}$ is orderable.
\end{proof}

\begin{remark}
\label{remark:filter-sp-ord}
A subfamily $\mathcal{P}$ of a filter $p$ on a set $X$ is called a  \textit{base for $p$} if for every $P \in p$ there exits $Q \in \mathcal{P}$ such that $Q \subset P$.
In \cite[Theorem 5.1]{garcia-gutev-nogura-sanchis-tomita:2002}, 
it was proved that 
a filter space $X_p$ is $p$-orderable if and only if $p$ has a nested base  $\mathcal{P}$ for $p$ satisfying $\bigcap\mathcal{P} = \emptyset$.
The implication
(c) $\Rightarrow$ (a) of Proposition \ref{proposition:filter-sp-regular-orderable} also follows from this theorem.
Indeed,
enumerate $X=\{ x_\alpha : \alpha < \kappa\}$ and let 
$P_\alpha =\{ x_\beta : \alpha \leq \beta <\kappa\}$ for $\alpha <\kappa$.
Then the family $\mathcal{P}=\{P_\alpha: \alpha<\kappa\}$ is a nested base of $p^\kappa$ such that $\bigcap\mathcal{P}=\emptyset$.
\end{remark}

\section{``Three'' examples of filter spaces}
\label{section:examples}

The following lemma is certainly known. We include its proof only for the reader's convenience.
\begin{lemma}
\label{discrete:lemma}
Every infinite discrete space $D$ is orderable by each of the two linear orders $\le_1$ and $\le_2$ on $D$  such that
$D$ has no $\le_1$-minimal element, while $D$ has a $\le_2$-minimal element.
\end{lemma}
\begin{proof}
Let $\mu=|D|$, and let $\leq_l$ be the lexicographical order on the product set $\mu \times \mathbb{Z}$. 

Since $D$ is infinite, $|D|=|\mu\times\mathbb{Z}|$ holds, and so 
there exists a bijection $f:D\to \mu\times\mathbb{Z}$.
Now we can define the required order $\le_1$ on $D$ by $x\le_1 y$ if and only if $f(x)\le_l f(y)$ for all $x,y\in D$.

Since the ordered subset $[(0, 0) ,\to)_{\leq_l}$ of $(\mu \times \mathbb{Z}, \leq_l)$ has size $\mu=|D|$, we can fix a bijection
$g:D\to [(0, 0) ,\to)_{\leq_l}$
and define 
the required order $\le_2$ on $D$ by $x\le_2 y$ if and only if $g(x)\le_l g(y)$ for all $x,y\in D$.
\end{proof}

\begin{example}
\label{example:a<psi<card}
Let $\kappa$, $\lambda$ and $\mu$ be infinite cardinals such that $\lambda$ is regular and $\kappa \leq \lambda \leq \mu$.
Then there exits a weakly orderable filter space $Z_p$ such that
$a(Z_p) =\kappa$, $\psi (Z_p)=\lambda$ and $|Z_p|=\mu$.
If, moreover, $\kappa$ is regular, 
then $Z_p$ can be taken to be an orderable space. 
\end{example}

\begin{proof}
Fix pairwise disjoint sets $X$, $Y$, $D$ such that
$|X|=\kappa$, $|Y|=\lambda$ and $|D|=\mu$.
Let $p^\kappa=\{A \subset X: |X\setminus A|<\kappa \}$ and $p^\lambda=\{B \subset Y: |Y\setminus B|<\lambda \}$
be the corresponding filters on $X$ and $Y$ respectively, and
let $p$ be the filter on $Z=X\cup Y\cup D$
generated by $\{ A \cup B : A \in p^\kappa, \, B\in p^\lambda\}$. 
We claim that $Z_p$ is the required filter space.

Since $p \in \overline{X}$ in $Z_p$, we have $a(Z_p)\leq |X|=\kappa$. For every $C \subset Z$ with $|C|<\kappa$, we have 
$X\setminus C \in p^\kappa$ and $Y\setminus C\in p^\lambda$,
which shows that $p \notin \overline{C}$. Therefore,
$a(Z_p)=\kappa$.
Since $\lambda$ is a regular cardinal and $\kappa \leq \lambda$, we have $\psi(Z_p)=
\max\{\psi (X_{p^\kappa}), \psi(Y_{p^\lambda}) \}=\lambda$ by Proposition \ref{a:psi:in:p(kappa)} (\ref{a:psi:in:p(kappa)-2}).
From $\kappa \leq \lambda \leq \mu$, we have $|Z_p|=\mu$.

Let us show that $Z_p$ is weakly orderable.
By Proposition \ref{theorem102:product-selection}, $X_{p^\kappa}$ is weakly $p^\kappa$-orderable by some linear order $\le_X$ such that $p^\kappa$ is $\leq_X$-maximal.
Since
$\lambda$ is regular,
Proposition \ref{proposition:filter-sp-regular-orderable} implies that
$Y_{p^\lambda}$ is $p^\lambda$-orderable by some linear order $\le_Y$  such that $p^\lambda$ is $\leq_Y$-minimal.
If $Y_{p^\lambda}$ has no $\leq_Y$-maximal element,
then we take as $\le_D$ the linear order $\le_1$ (on $D$) defined in Lemma \ref{discrete:lemma}.
If $Y_{p^\lambda}$ has a $\leq_Y$-maximal element, 
then we take as $\le_D$ the linear order $\le_2$ (on $D$) defined in Lemma \ref{discrete:lemma}.
Then $\leq_D$ induces the discrete topology on $D$, and $D$ has a $\leq_D$-minimal element if and only if $Y_{p^\lambda}$ has a $\leq_Y$-maximal element.

Let $\preceq$ be the linear order on $Z\cup\{p\}$ such that 
the restrictions of $\preceq$ to $X$, $Y$ 
and $D$ coincide with $\leq_X$, $\leq_Y$ and $\leq _D$, respectively,
and $x\prec p\prec y \prec d$ for every $x \in  X$, $y \in Y$ and $d \in D$.
Then the order topology induced by $\preceq$ is coarser than the topology of $Z_p$, which shows that $Z_p$ is weakly orderable.

If $\kappa$ is regular, then $X_{p^\kappa}$ is orderable by $\leq_X$; see Proposition \ref{proposition:filter-sp-regular-orderable}.
Therefore,
$Z_p$ is orderable by $\preceq$ because $D$ has a $\leq_D$-minimal element if and only if $Y_{p^\lambda}$ has a $\leq_Y$-maximal element.
\end{proof}

 \begin{example}
\label{example:psi<a<card}
Let $\kappa$, $\lambda$ and $\mu$ be infinite cardinals such that $\kappa$ is regular and $\kappa \leq \cf(\lambda) \leq \lambda \leq \mu \leq \nu$.
Then there exits a weakly $p$-orderable filter space $Z_p$ such that
$\psi(Z_p) =\kappa$, $a (Z_p)=\lambda$, $\| p\| =\mu$ and $|Z_p|=\nu$.
\end{example}
\begin{proof}
Let $X=\kappa \times \mu$
and 
let $D$ be a set disjoint from $X$ such that
$|D|=\nu$.
Let $Z=X\cup D$.
Consider the filter  $p$ on $Z$ generated by $\{((\kappa \setminus \beta) \times \mu ) \setminus A : \beta < \kappa ,\, A\subset X ,\, |A|<\lambda\}$.
Then $Z_p$ is the desired filter space.

For each $\beta <\kappa$, let $P_\beta = (\kappa\setminus\beta) \times \mu$.
Then $\{P_\beta : \beta < \kappa \}$ is a nested subfamily of $p$ such that $\bigcap_{\beta < \kappa} P_\beta = \emptyset$, 
so $Z_p$ is weakly $p$-orderable 
by Theorem \ref{lemma:filter-sp-suff-cond-wo}.
Furthermore,
$\psi (Z_p) \leq \kappa$ by Proposition \ref{proposition:p-ord->psi<a}.

Suppose that $\psi (Z_p) < \kappa$. Then there exists
$\mathcal{U} \subset p$ 
such that
$\bigcap\mathcal{U}=\emptyset$ and $|\mathcal{U}|<\kappa$.
For every $U\in \mathcal{U}$, take $\beta_U < \kappa$ and $A_U \subset X$ so that $|A_U|<\lambda$ and $((\kappa \setminus \beta_U ) \times \mu ) \setminus A_{U} \subset U$.
Since $|\mathcal{U}|<\kappa\le\cf(\lambda)$ and $|A_U|<\lambda$
for every $U\in\mathcal{U}$, 
the set $A=\bigcup_{U \in \mathcal{U}} A_U $
satisfies
$|A|<\lambda$. Since $\lambda\le\mu$, there exists 
$\gamma<\mu$ such that $(\kappa\times\{\gamma\})\cap A=\emptyset$.
Since $|\mathcal{U}|<\kappa=\cf(\kappa)$ and $\beta_U<\kappa$ for every $U\in\mathcal{U}$, we have
$\beta=\sup\{\beta_U : U \in \mathcal{U} \}<\kappa$.
Now $(\beta+1,\gamma)\in ((\kappa \setminus \beta_U ) \times \mu ) \setminus A_{U} \subset U$ for every $U\in\mathcal{U}$,
in contradiction with $\bigcap\mathcal{U}=\emptyset$.
This contradiction shows that $\psi (Z_p)=\kappa$. 

Let $C=\kappa \times \lambda$.
Then $|C| = \lambda$ since $\kappa \leq \lambda$, 
and $p \in \overline{C}$ 
since $|C\cap (\kappa \setminus \beta) \times \mu|=\lambda$ for every $\beta <\kappa$. 
Hence $a(Z_p) \leq |C|=\lambda$.
If $S \subset Z$ with $|S|< \lambda$,
then $X\setminus S \in p$, which implies $p \notin \overline{S}$.
This shows that $a(Z_p) \geq \lambda$, and hence $a(Z_p)=\lambda$.

Clearly, $\| p\| =\mu$.
Finally,
since $\kappa \leq \lambda \leq \mu \leq \nu$, we have $|Z_p|=\nu$.
\end{proof}  
   
\begin{example}
\label{exotic}
Let $\kappa$  be an infinite cardinal such that $\omega_1 \leq \kappa$.
Then there exists a filter space $Z_p$ such that $\psi(Z_p)=\omega_1$, $a(Z_p)=\kappa$ and $Z_p$ does not admit a (separately) continuous weak selection.
\end{example}

\begin{proof}
Let $X$ and $Y$ be sets such that $|X|=\omega_1$ and $|Y|=(2^\kappa)^+$. Define
$Z=X\times Y$, and let
$\pi : Z \to Y$ be the projection.
Let 
$p$ be the filter on $Z$ generated by the family 
$\{P_0 \times P_1 : P_0 \in p^{\omega},\, P_1 \in p^\kappa \}$, where  $p^{\omega}  = \{A \subset X : |X \setminus A|<\omega \}$ and $p^\kappa =\{ B \subset Y : |Y\setminus B| < \kappa \}$.
We show that $Z_p$ is the required filter space.

Note that
$P_x = (X\setminus \{x\}) \times Y\in p$
for every $x \in X$ and $\bigcap_{x \in X} P_x = \emptyset$, 
so
$\psi(Z_p)\leq |X|=\omega_1$.
To show $\psi(Z_p)\geq \omega_1$, let $\mathcal{U} \subset p$ with $\bigcap\mathcal{U} =\emptyset$. For each $U \in \mathcal{U}$, take $A_U \subset X$ and $B_U \subset Y$ such that $|A_U|<\omega$, $|B_U|<\kappa$  and $(X\setminus A_U) \times (Y \setminus B_U) \subset U$.
Since $\bigcap\mathcal{U}=\emptyset$, we have $\bigcup_{U\in \mathcal{U}}A_U=X$ or $\bigcup_{U\in\mathcal{U}}B_U =Y$,
which implies $|\mathcal{U}|\geq \omega_1$ since $|X|=\omega_1\leq \kappa<|Y|$.
Hence $\psi(Z_p)\geq |\mathcal{U}|\ge\omega_1$.

Let us show that $a(Z_p)=\kappa$.
Let $C \subset Z$ with $|C|<\kappa$.
Then the set $P=X  \times (Y \setminus \pi (C))$ is contained in $p$ and $C \cap P =\emptyset$, which shows that 
$p\not\in\overline{C}$. This implies
$a(Z_p) \geq \kappa$.
Next, let $K\subset Y$ with $|K|=\kappa$. Then $(X\times K) \cap P \ne\emptyset$ for each $P \in p$; that is, $p\in\overline{X\times K}$. Since $\omega_1\leq \kappa$, we have $a(Z_p)\leq |X \times K| =\kappa$.

It remains to show that $Z_p$ does not admit a separately continuous weak selection.
Suppose for contradiction that $\sigma$ 
is a separately continuous weak selection on 
$Z_p$.
Define
\begin{equation}
\label{L_x:R_x}
L_x=\pi((\leftarrow,p)_{\preceq_\sigma}\cap(\{x\}\times Y))
\text{  and  } 
R_x=\pi((p,\rightarrow)_{\preceq_\sigma}\cap(\{x\}\times Y))
\end{equation}
for each $x \in X$, and 
\begin{equation}
\label{L:L_x}
L=\{x \in X:|L_x|=|Y|\}
\text{  and  } 
R=\{x \in X:|R_x|=|Y|\}.
\end{equation}
Since $L_x\cup R_x=Y$ for each $x \in X$, either $|L|=\omega_1$ or $|R|=\omega_1$.
Without loss of generality, we may assume $|L|=\omega_1$.

For each $x \in L$ and $y \in L_x$, we have $(x,y) \prec_\sigma p$ by \eqref{L_x:R_x}, and since $\sigma$ is separately continuous,
there exist $A_{x,y} \subset X$ and $B_{x,y} \subset Y$ such that 
$|A_{x,y}|<\omega$, $|B_{x,y}|<\kappa$ and
\begin{equation}
\label{eq:7}
(X\setminus A_{x,y})\times (Y\setminus B_{x,y}) \subset ((x,y), \rightarrow)_{\preceq_\sigma}.
\end{equation}

Fix
$x\in L$. Since $|\{A_{x,y} : y \in L_x \}|\leq |X^{<\omega}|=\omega_1<|Y|=|L_x|$ by \eqref{L:L_x}
and $|Y|$ is regular, there exist $Y_x \subset L_x$ and $A_x\subset X$ such that $|Y_x|=|Y|$ and 
\begin{equation}
\label{eq:*}
A_{x,y} =A_x
\text{ for every }y \in Y_x.
\end{equation} 
(For a set $X$ and a cardinal $\tau$, we use $X^{<\tau}$ to denote the set of all subsets of $X$ of cardinality smaller than $\tau$.)
Since $|Y_x|=|Y|=(2^\kappa)^+$, we have $\lambda^{<\kappa}<|Y_x|$ for each $\lambda <|Y_x|$. Thus, we can apply the $\Delta$-system lemma
\cite[Chapter II, Theorem 1.6]{kunen:1983} to the family $\{B_{x,y} :y \in Y_x \}$ to  find
$S_x \subset Y_x$ and $B_x \subset Y$ such that $|S_x|=|Y_x|=|Y|$ and 
\begin{equation}
\label{eq:**}
B_{x,y} \cap B_{x,y'}=B_x
\text{ for all }
y,y' \in S_x
\text{ with }y\ne y'.
\end{equation}
Finally,
note that $|A_x|<\omega$ and $|B_x|<\kappa$. 

Applying the $\Delta$-system lemma again to the family $\{A_x :x \in L \}$ of finite subsets of $X$, we get $T\subset L$ and $A\subset X$ such that $|T|=\omega_1$ and 
\begin{equation}
\label{eq:***}
A_x \cap A_{x'}=A
\text{ for all }
x, x' \in T
\text{ with }x\ne x'.
\end{equation}
Since $|A|<\omega <\omega_1=|L|$,
we can find $x_0 \in L\setminus A$.
Since
$|\bigcup_{x \in T}B_{x}|\leq \kappa <|Y|=|L_{x_0}|$,
we can fix  $y_0 \in L_{x_0}\setminus \bigcup_{x \in T}B_{x}$.

Since $x_0 \notin A$, it follows from \eqref{eq:***} that $x_0$ is contained in at most one member of the family $\{A_x: x\in T\}$, which implies  
$|\{x \in T : x_0 \notin A_x\}|=\omega_1
>|A_{x_0,y_0}|
$. Hence there exists $x_1 \in T\setminus A_{x_0,y_0}$ such that $x_0 \notin A_{x_1}$.
Since $y_0 \notin \bigcup_{x \in T}B_{x}$ and $x_1 \in T$,
we have $y_0 \notin B_{x_1}$.
It follows from this and \eqref{eq:**} that
$y_0$ is contained in at most one member of the family $\{B_{x_1,y}: y\in S_{x_1} \}$, which implies
$|\{y \in S_{x_1} : y_0 \notin B_{x_1, y}\}|=|S_{x_1}|
=|Y|>|B_{x_0,y_0}|$. Therefore, we can choose 
$y_1 \in S_{x_1} \setminus B_{x_0,y_0}$ such that 
$y_0 \notin B_{x_1,y_1}$.
Since $y_1 \in S_{x_1}\subset Y_{x_1}$, 
we have
$x_0 \notin A_{x_1} = A_{x_1,y_1}$ by \eqref{eq:*}.

Since $x_0 \in L$ and $y_0 \in L_{x_0}$, 
$(x_1,y_1) \in (X\setminus A_{x_0,y_0} )\times (Y\setminus B_{x_0,y_0}) \subset ((x_0,y_0),\rightarrow)_{\preceq_\sigma}$ by \eqref{eq:7}, which implies $(x_0,y_0) \prec_\sigma (x_1,y_1)$. 
On the other hand,
since $x_1 \in T \subset L$ and $y_1 \in Y_{x_1} \subset L_{x_1}$,
$(x_0,y_0) \in (X\setminus A_{x_1,y_1} )\times (Y\setminus B_{x_1,y_1}) \subset ((x_1,y_1),\rightarrow)_{\preceq_\sigma}$
 by \eqref{eq:7}, 
so $(x_1,y_1) \prec_\sigma (x_0,y_0)$, giving a contradiction. 
\end{proof}
  
\section{Weak orderability of products of filter spaces}
\label{section:produts-of-filter-spaces}

In this section, we establish sufficient conditions for the 
weak orderability of products of filter spaces. 

For a selection relation $\preceq_\sigma$ on a space $X$ and $A, B \subset X$,
we write $A \prec_\sigma B$ if $x\prec_\sigma y$ for every $x \in A$ and $y \in B$.

We start with the following lemma generalizing \cite[Lemma 2.4]{garcia-miyazaki-nogura:2013}.

\begin{lemma}
\label{lemma103:product-selection}
Let $\gamma$ be an ordinal and 
 $Z$ a space with a point $p\in Z$ and a family $\{W_\alpha : \alpha < \gamma \}$ of subsets of $Z$ 
satisfying the following conditions:
\begin{enumerate}[{\rm (i)}]
\item \label{103:1} each $W_\alpha$ is  open in $Z$,
\item \label{103:3} $W_\alpha\cap W_\beta=\emptyset$ whenever $\alpha,\beta<\gamma$ and $\alpha\not=\beta$,
\item \label{103:4} $\bigcup_{\alpha<\gamma} W_\alpha=Z\setminus\{p\}$,
\item \label{103:5} $p\not\in \overline{\bigcup_{\beta<\alpha}W_{\beta}}$
for every $\alpha<\gamma$,
\item \label{103:2} each $W_\alpha$ admits a continuous weak selection $\sigma_\alpha$.
\end{enumerate}

Then $Z$ 
has a continuous weak selection.
Moreover, if each $W_\alpha$ is weakly orderable, then
$Z$ is weakly $p$-orderable.  
\end{lemma}

\begin{proof}
It follows from (\ref{103:3}) and (\ref{103:4}) that $\{p\}\cup\{  W_\alpha: \alpha <\gamma \}$ is a partition of $Z$.
Thus, there exists a unique weak selection $\sigma$ on $Z$ satisfying the following conditions:
\begin{enumerate}[(a)]
\item 
$\sigma\restriction_{\mathcal{F}_2(W_\alpha)}=\sigma_\alpha$ 
for every $\alpha<\gamma$,
\item $W_\beta\prec_\sigma W_\alpha$ whenever $\beta<\alpha<\gamma$,
\item $W_\alpha\prec_\sigma \{p\}$ for all $\alpha<\gamma$.
\end{enumerate}

To show that $\sigma$ is 
continuous,
for fixed
$x,y\in Z$ with $x\prec_{\sigma} y$,
we need to find an open neighborhood $U_x$ of $x$
and 
an open neighborhood $U_y$ of $y$
such that
$U_x\prec_\sigma U_y$; see \cite[Theorem 2.6]{gutev-nogura:2010}.
We shall consider three cases.
	
{\em Case 1.\/}
{\sl $x,y\in W_{\alpha}$ for some $\alpha<\gamma$.}
Then $x\prec_{\sigma_\alpha} y$ by (a). Since $\sigma_\alpha$ 
is continuous by (\ref{103:2}),
there exist open subsets $U_{x}$ and $U_{y}$ of $W_{\alpha}$ such that $x\in U_{x},y\in U_{y}$ and $U_{x}\prec_{\sigma_{\alpha}}U_{y}$.
Applying (a) once again, we conclude that $U_{x}\prec_{\sigma}U_{y}$.
Since $W_\alpha$ is open in $Z$ by (i), both $U_x$ and $U_y$ are open in $Z$ as well. 	

{\em Case 2\/}. {\sl $x\in W_\beta$ and $y\in W_\alpha$ for 
some $\beta<\alpha<\gamma$\/}.
Then $W_\beta\prec_{\sigma} W_\alpha$ by (b).
Both $W_\alpha$ and $W_\beta$ are open in $Z$ by (\ref{103:1}), and
 we can let $U_x=W_\beta$ and $U_y=W_\alpha$.

{\em Case 3\/}. {\sl $x\in W_\alpha$ for some $\alpha<\gamma$ and $y=p$.\/}
It follows from (\ref{103:1}) and $x\in W_\alpha$ that $U_x=\bigcup_{\beta<\alpha+1} W_\beta$ is an open neighborhood of $x$ in $Z$.
By (\ref{103:5}), $U_y=Z\setminus \overline{\bigcup_{\beta<\alpha+1}W_{\beta}}
=Z\setminus\overline{U_x}$ is an open neighborhood of $y=p$ in $Z$.
Since $U_x=\bigcup_{\beta<\alpha+1} W_\beta$
and 
$U_y\subset\{p\}\cup \bigcup_{\beta\ge \alpha+1} W_\beta$, it follows from (b) and (c) that $U_x\prec_\sigma U_y$.

If each $W_\alpha$ is weakly orderable, then the same argument as above and \cite[Theorem 2.5]{gutev-nogura:2010} show that  $Z$ is weakly orderable by a linear order $\preceq$ and $p$ is $\preceq$-maximal.
\end{proof}

\begin{remark}
The condition (\ref{103:5}) in Lemma \ref{lemma103:product-selection} cannot be dropped.
Indeed, let $X=\{x_\alpha : \alpha < \omega_1 \}$ be a set with $|X|=\omega_1$ and $p^\omega=\{A \subset X : |X\setminus A|<\omega \}$.
For each $\alpha < \omega_1$, let $W_\alpha =\{x_\alpha \}$.
Then the point $p^\omega$ in the filter space $X_{p^\omega}$ and the family $\{W_\alpha: \alpha<\omega_1 \}$ satisfy the conditions (\ref{103:1})--(\ref{103:4}) and (\ref{103:2}) in Lemma \ref{lemma103:product-selection}, yet
$X_{p^\omega}$ does not have a continuous weak selection
by 
Proposition \ref{theorem102:product-selection}.
\end{remark}           

\begin{lemma}
\label{lemma104:product-selection}
Let $\gamma$ be an ordinal and let $X$ and $Y$ be spaces.
Assume that points $p\in X$ and $q \in Y$ and families $\mathcal{U}=\{U_{\alpha}:\alpha<\gamma\}$ of subsets of $X$ and $\mathcal{V}=\{V_{\alpha}:\alpha<\gamma\}$ of subsets of $Y$
satisfy the conditions {\rm (i)--(iv)} in Lemma \ref{lemma103:product-selection} and the following condition:
\begin{enumerate}[{\rm (i)'}]
\setcounter{enumi}{4}
\item both $U_\alpha\times Y$ and $X \times V_\alpha$ admit continuous weak selections for every $\alpha<\gamma$.
\end{enumerate}
Then $X\times Y$ admits a continuous weak selection.
Moreover, if both $U_\alpha\times Y$ and $X \times V_\alpha$
are weakly orderable for every $\alpha<\gamma$, then 
$X\times Y$ is weakly $(p,q)$-orderable.
\end{lemma}
 
\begin{proof}
Note that each $U_\alpha$ and $V_\alpha$ are clopen subsets of $X$ and $Y$, respectively.
For every $\alpha < \gamma$, let
$W_\alpha =(U_\alpha \times (Y \setminus \bigcup _{\beta<\alpha}V_\beta)) \cup ((X\setminus \bigcup _{\beta < \alpha } U_\beta) \times V_\alpha).$
Then $\{W_\alpha : \alpha < \gamma \}$ is a pairwise disjoint family of open subsets of $X\times Y$ such that 
$\bigcup_{\alpha <\gamma }W_\alpha =X\times Y\setminus \{(p,q)\}$
and $(p,q) \notin \overline{\bigcup_{\beta<\alpha}W_{\beta}}$
for every $\alpha <\gamma$.

Let $\alpha < \gamma$ be arbitrary.
We
claim that $W_\alpha$
admits a continuous weak selection.
To show this, note that
$O_0=U_\alpha \times (Y \setminus \bigcup _{\beta<\alpha}V_\beta) $
and $O_1=(X\setminus \bigcup _{\beta \leq \alpha } U_\beta) \times V_\alpha$
are disjoint open subsets of $W_\alpha$ 
such that $W_\alpha =O_0 \cup O_1$.
Let 
$\sigma_0$ and $\sigma_1$ be continuous weak selections of $U_\alpha \times Y$ and $X\times V_\alpha$, respectively.
Then the weak selection $\sigma$ on $W_\alpha$
defined by
$O_0\preceq_\sigma O_1$ and 
$\sigma\restriction_{\mathcal{F}_2(O_i)}=\sigma_i$ 
for $i\in \{0,1\}$
is continuous.  
Thus, 
$X\times Y$ admits a continuous weak selection
by Lemma \ref{lemma103:product-selection}.
Furthermore,
if both $\sigma_0$ and $\sigma_1$ above are induced by linear orders,
then $W_\alpha$ is weakly orderable by the linear order $\preceq_\sigma$.
Hence $X\times Y$ is weakly $(p,q)$-orderable by Lemma \ref{lemma103:product-selection}. 
\end{proof}
 
\begin{lemma}
\label{lemma:filter-sp-suff-cond-prod-wo}
For $Z\in\{X,Y\}$, let $p_Z$ be a filter on a set $Z$ with a 
nested subfamily  $\{P^Z_\alpha : \alpha < \gamma \} \subset p_Z$ satisfying $P^Z_\alpha \subsetneq \bigcap_{\beta<\alpha} P^Z_\beta$ for every  $\alpha < \gamma$ and 
	$\bigcap_{\alpha < \gamma} P^Z_\alpha = \emptyset$.
	Then $X_{p_X}\times Y_{p_Y}$ is  weakly $(p_X,p_Y)$-orderable.
\end{lemma}

\begin{proof} 
For $Z \in \{X,Y\}$ and $\alpha<\gamma$,
let $U_\alpha^Z = (\bigcap_{\beta <\alpha} P^Z_\beta) \setminus P^Z_\alpha $, where $U_0^Z=Z \setminus P^Z_0  $.
For $Z \in \{X,Y\}$, we claim that
the point $p_Z \in Z_{p_Z}$ and 
the family $\{U_\alpha^Z :\alpha < \gamma  \}$, 
satisfy the conditions 
(\ref{103:1})--(\ref{103:5}) in Lemma \ref{lemma103:product-selection} and (v)' in Lemma \ref{lemma104:product-selection}.
Indeed, (\ref{103:1}) and (\ref{103:3}) are clear.
(\ref{103:4}) follows from $\bigcap_{\alpha < \kappa} P_\alpha ^Z= \emptyset$.
For (\ref{103:5}), let $\alpha<\gamma$. 
Since $\bigcup_{\beta < \alpha}U_\beta^Z \cap P_\alpha^Z  = \emptyset $ and  $ P_\alpha ^Z\in p_Z$, we have $p_Z \notin \overline{\bigcup_{\beta <\alpha}U_\beta^Z}$.
It remains to show (\ref{103:2})'. Since $Y_{p_Y}$  is weakly orderable and 
$\{ \{x\} \times Y_{p_Y} : x \in U_\alpha^X \}$ is a partition into clopen subsets of $U^X_\alpha \times Y_{p_Y}$, 
the product
$U^X_\alpha \times Y_{p_Y}$ is weakly orderable for $\alpha<\gamma$.  
Similarly, $X_{p_X} \times U^Y_\alpha $ is weakly orderable for $\alpha< \gamma$. 
Thus, the conclusion follows from Lemma \ref{lemma104:product-selection}.
\end{proof}

The following fact follows from \cite[Proposition 5.4]{garcia-gutev-nogura-sanchis-tomita:2002}. We give a direct proof here for the sake of completeness.

\begin{fact}\label{fact:partition-wo}
For every nested family $\mathcal{P}$ of subsets of a set $X$ with 
$\bigcap\mathcal{P} =\emptyset$, there exist an 
ordinal $\gamma<|\mathcal{P}|^+$ and $\{P_\alpha :\alpha < \gamma  \} \subset \mathcal{P}$ such that 
$P_\alpha \subsetneq \bigcap_{\beta<\alpha} P_\beta$ 
for every  $\alpha < \gamma$ and $\bigcap_{\alpha < \gamma}P_\alpha = \emptyset$.  
\end{fact} 

\begin{proof}
The family is constructed by transfinite induction.
Since $\bigcap\mathcal{P}=\emptyset$, we can take
$P_0 \in \mathcal{P}$.
Assume $P_\alpha \in \mathcal{P}$ has been taken for each $\alpha< \delta$ in such a way that $P_\alpha \subsetneq \bigcap_{\beta<\alpha} P_\beta$ for every  $\alpha < \delta$.
If $\bigcap_{\alpha<\delta}P_\alpha = \emptyset$, then $\{P_\alpha :\alpha < \delta  \}$ is as required (if one takes $\gamma=\delta$).
Suppose that
$\bigcap_{\alpha<\delta}P_\alpha \ne \emptyset$.
Then we can use Lemma \ref{element:in:the:intersection}
to select $P_\delta\in \mathcal{P}$ such that
$P_\delta\subsetneq \bigcap_{\alpha<\delta}P_\alpha$.
According to this construction and $\bigcap \mathcal{P} = \emptyset$, we get $\bigcap_{\alpha<\gamma}P_\alpha = \emptyset$ for some $\gamma < |\mathcal{P}|^+$. 
\end{proof}

\begin{corollary}
\label{theorem:filter-sp-suff-cond-prod-wo}
If $X_p$ a weakly $p$-orderable filter space,
then $X_p \times X_p$ is weakly $(p,p)$-orderable.
\end{corollary}

\begin{proof}
It follows from  Theorem \ref{lemma:filter-sp-suff-cond-wo}, Fact \ref{fact:partition-wo} and Lemma \ref{lemma:filter-sp-suff-cond-prod-wo}.
\end{proof}

\begin{corollary}
\label{theorem4:product-selection}
If $X_{p}$ and $Y_{q}$ are filter spaces such that $a(X_{p})=\| p\| =a(Y_{q})=\| q\| $,
then $X_{p}\times Y_{q}$ is weakly $(p,q)$-orderable.
\end{corollary}
\begin{proof}
Let $X_{p}$ and $Y_{q}$ be filter spaces such that $a(X_{p})=\| p\| =a(Y_{q})=\| q\|=\kappa$.
By Proposition \ref{corollary:dispersion-character},
for  $Z \in \{ X,Y\}$,
there exists a nested subfamily  $\{P^Z_\alpha : \alpha < \kappa \} \subset p_Z$ satisfying 
$P^Z_\alpha \subsetneq \bigcap_{\beta<\alpha} P^Z_\beta$ for every  $\alpha < \kappa$ and 
$\bigcap_{\alpha < \kappa} P^Z_\alpha = \emptyset$.
Now the conclusion follows from Lemma
\ref{lemma:filter-sp-suff-cond-prod-wo}.
\end{proof}

\begin{corollary}
\label{corollary4:product-selection}
If $p$ and $q$ are filters on a set $X$ such that $a(X_{p})=a(X_{q})=|X|$,
then $X_{p}\times X_{q}$ is weakly $(p,q)$-orderable.
\end{corollary}
\begin{proof}
From the assumption of our corollary and Proposition \ref{basic:proposition:about:|p|}, we have 
 $a(X_{p})=\| p\| =|X|=a(Y_{q})=\| q\|$, and the conclusion follows from Corollary \ref{theorem4:product-selection}.
\end{proof}

From
Proposition \ref{proposition:ultrafilter-dispersion-ch} and
Corollary \ref{theorem4:product-selection}, we obtain
the following

\begin{corollary}
\label{proposition5:product-selection}
If $p$ and $q$ are  ultrafilters on a set $X$ with $\|p\| =\|q\|$,
then $X_{p}\times X_{q}$ is weakly $(p,q)$-orderable.
\end{corollary}
   
Our next theorem shows that
the assumption
$a(X_p)=|X_p|$ in Corollary \ref{corollary4:product-selection}
is necessary in a certain sense.
          
\begin{theorem}
\label{theorem5:product-selection}
Let $X$ be a non-discrete space with $a(X)<|X|$.
Then:
\begin{enumerate}[{\rm (i)}]
\item
There exists a weakly orderable filter space $Y_q$  such that $a(Y_q) =a(X)$, $|Y_q|=|X|$ and 
$X\times Y_q$ does not admit a separately continuous weak selection.
\item 
If $|X|$ is regular, then there exists an orderable filter space
$Y_q$ such that $a(Y_q)=|Y_q|=|X|$ 
and $X\times Y_q$ does not admit a separately continuous weak selection.
\end{enumerate}
\end{theorem} 
\begin{proof}
There exists a weakly orderable filter space $Y_q$ satisfying 
$a(Y_q)=a(X)$, $\psi(Y_q)=a(X)^+$ and $|Y_q|=|X|$; see Example \ref{example:a<psi<card}.
By Theorem \ref{corollary101:product-selection} and $\psi(Y_q) >a(X)$, $X\times Y_q$ does not admit a separately continuous weak selection, which proves (i).
For (ii), it suffices to take a filter space $Y_q$ such that 
$a(Y_q)=\psi(Y_q)=|Y_q|=|X|$; see Example \ref{example:a<psi<card}.
  \end{proof}   

On the other hand, 
our next remark demonstrates that
the inequality $a(Z_p)<|Z_p|$ is not an obstacle for the square $Z_p\times Z_p$ of a filter space $Z_p$ to have a continuous weak selection.

\begin{remark}
For a filter space $Z_p$ from Example \ref{example:psi<a<card},
$Z_p \times Z_p$ is weakly $(p,p)$-orderable by 
Corollary \ref{theorem:filter-sp-suff-cond-prod-wo}.
\end{remark}

\begin{remark}
\label{remark:c-cpt-in-Gutev's-thm}
In Theorem \ref{thm:Gutev}, the assumption that $X$ is countably compact cannot be dropped. Indeed, by Example \ref{example:psi<a<card}, there exists a weakly $p$-orderable filter space $Z_p$ such that $\psi(Z_p)=\omega_1$.
By 
Corollary \ref{theorem:filter-sp-suff-cond-prod-wo}, $Z_p\times Z_p$ is weakly $(p,p)$-orderable (and thus, weakly orderable), while $Z_p$ is not metrizable since $\psi(Z_p)=\omega_1$.
\end{remark}

By applying Lemma \ref{lemma103:product-selection} and Fact \ref{fact:partition-wo}, we can get 
an alternative
proof of Theorem \ref{lemma:filter-sp-suff-cond-wo}.

\begin{proof}[Another proof of Theorem \ref{lemma:filter-sp-suff-cond-wo}]
To show the ``if'' part, 
let  $\mathcal{P}$ be a nested subfamily of $p$ with $\bigcap\mathcal{P} = \emptyset$.
By  Fact \ref{fact:partition-wo}, there are an ordinal $\gamma$ and $\{P_\alpha :\alpha < \gamma  \} \subset \mathcal{P}$ such that 
$P_\alpha \subsetneq \bigcap_{\beta<\alpha} P_\beta$ 
for every  $\alpha < \gamma$ and $\bigcap_{\alpha < \gamma}P_\alpha = \emptyset$.  
For  $\alpha<\gamma$, let $W_\alpha= (\bigcap_{\beta <\alpha} P_\beta) \setminus P_\alpha $, where $W_0= 
X \setminus P_0  $.
Then the point $p \in X_{p}$ and the family $\{W_\alpha :\alpha < \gamma  \}$ satisfy the conditions (i)--(v) in Lemma \ref{lemma103:product-selection}.
Indeed, (\ref{103:1}) and (\ref{103:3}) are clear, and (\ref{103:4}) and (\ref{103:5}) can be shown by the same argument as in the proof of Lemma \ref{lemma:filter-sp-suff-cond-prod-wo}.
(\ref{103:2}) follows from the fact that each $W_\alpha$ is a discrete subspace of $X_p$.
Thus, $p$ and  $\{W_\alpha :\alpha < \gamma  \}$ satisfy the conditions (i)--(v) in Lemma \ref{lemma103:product-selection}, and hence $X_p$ is weakly $p$-orderable.
  
To show the ``only if'' part,
assume that $X_p$ is weakly $p$-orderable by a linear order $\preceq$ such that $p$ is $\preceq$-maximal.
Then $(x,p)_\preceq \in p$ since $(x,p]_\preceq$ is an open neighborhood of $p$. Thus $\{(x,p)_\preceq : x \in X \}$ is a nested subfamily of $p$ satisfying $\bigcap_{x \in X} (x,p)_\preceq =\emptyset$.
\end{proof}

\section{Separately continuous weak selections on products of suborderable spaces}
\label{section:necessary-conditions}

Let us 
recall
a
theorem of Bula \cite{bula:1986}
which 
generalizes
a result of Engelking and Lutzer~\cite{engelking-lutzer:1977}.

\begin{theorem}[{\cite[Corollary 4]{bula:1986}}]
\label{theorem0:product-selection}
Let $X$ be an image of a suborderable space under a closed continuous map.
Then $X$ is not hereditary paracompact if and only if it contains a subset which is homeomorphic to a stationary subset of an uncountable regular cardinal.
\end{theorem}

\begin{lemma}
\label{proposition2:product-selection}
Let $S$ be a stationary subset of an uncountable regular cardinal $\kappa$.
Then $a(S)<\kappa$.
\end{lemma}
\begin{proof}
Suppose for contradiction that $a(S)=\kappa$.
Let  $\alpha \in S$.
Then $[0,\alpha)_\leq\cap S$ is closed in $S$ since $|[0,\alpha)_\leq\cap S|<\kappa =a(S)$ by our assumption.
Therefore, there exists $h(\alpha )<\alpha$ such that 
$[0,\alpha)_\leq\cap S \cap (h(\alpha),\alpha]_\leq =\emptyset$.
Then $(h(\alpha),\alpha]_\leq\cap S=\{\alpha\}$.
This shows that there exits a map $h : S \to \kappa$ such that 
$h(\alpha) < \alpha $ and $(h(\alpha),\alpha]_\leq\cap S=\{\alpha\}$. 
By the pressing down lemma \cite[Chapter II, Lemma 6.15]{kunen:1983}, 
there exist a stationary subset 
$S'\subset S$ and an ordinal $\beta<\kappa$ such that $h(\alpha)=\beta$ for all $\alpha\in S'$.
We can take $\alpha_1,\alpha_2\in S'$ with $\beta < \alpha_1<\alpha_2$.
Then $\alpha_1 \in (\beta , \alpha_2] \cap S = (h(\alpha_2), \alpha_2] \cap S =\{\alpha_2\}$, which contradicts $\alpha_1 \ne \alpha_2$.
\end{proof}

The following theorem 
removes the superfluous assumption 
$a(X)\leq \omega$ 
from
\cite[Theorem 3.1]{garcia-miyazaki-nogura:2013}.

\begin{theorem}
\label{theorem:necessary-cond-hered-para}
Let $X$ be an image of a suborderable space under a closed continuous map such that $X\times Y$ admits a separately continuous weak selection for some non-discrete space $Y$. Then $X$ is hereditarily paracompact.
\end{theorem}  

\begin{proof}
Suppose for contradiction that $X$ is not hereditarily paracompact. 
By Theorem \ref{theorem0:product-selection}, $X$ contains a subset which is homeomorphic to a stationary subset $S$ of an uncountable regular cardinal $\kappa$.
Since $X\times Y$ has a separately continuous weak selection for some non-discrete space $Y$ by our assumption,
its subspace $S\times Y$ also admits a separately continuous weak selection.
On the other hard, since $a(S)<\kappa$ by Lemma \ref{proposition2:product-selection}, 
Theorem \ref{theorem1:product-selection} implies that $S\times Y$ does not admit a separately continuous weak selection.
\end{proof}

From Theorems \ref{corollary101:product-selection}, \ref{tot:disc}
and \ref{theorem:necessary-cond-hered-para},
we get the following

\begin{corollary}
\label{corollary112:product-selection}
Let $X$ be a non-discrete image of a suborderable space under a closed continuous map. If $X\times X$ admits a separately continuous weak selection, then $X$ is hereditarily paracompact, totally disconnected 
and satisfies the inequality
$\psi(X)\leq a(X)$.
\end{corollary}  

In the next two remarks we use an observation that every filter space is totally disconnected (even zero-dimensional) and hereditarily paracompact.

\begin{remark}
\label{question50:product-selection}
The converse of Corollary \ref{corollary112:product-selection} does not hold.
Indeed, there exists a hereditarily paracompact totally disconnected space $Z_p$ such that $\psi (Z_p) \leq a(Z_p)$, yet
$Z_p$ does not admit a separately continuous weak selection; see Example \ref{exotic}.
\end{remark} 

\begin{remark}
\label{example:necessary-cond-examples}
Hereditary paracompactness, total disconnectedness and the
inequality
$\psi(X) \leq a(X)$ are independent of each other in the realm of orderable spaces.
Indeed, 
$\omega_1$ is a totally disconnected non-(hereditarily) paracompact space such that $\psi(\omega_1)=a(\omega_1)=\omega$.
The real line $\mathbb{R}$ is a hereditarily paracompact non-totally disconnected space such that $\psi(\mathbb{R}) =a(\mathbb{R})=\omega$.
There is 
a totally disconnected, hereditarily paracompact 
orderable
space $Z_p$
satisfying
$a (Z_p) <\psi(Z_p)$; see Example \ref{example:a<psi<card}.
\end{remark}

\begin{corollary}
\label{corollary102:product-selection}
If $X$ is a non-discrete orderable space such that
$X\times X$ admits a separately continuous weak selection, then $\psi (X)=a(X)$.
\end{corollary}
\begin{proof}
Since $X$ is orderable, 
$\psi (X)=t(X)$; see \cite[3.12.4 (d)]{engelking:1989}.
Since $a(X)\le t(X)$ by \eqref{a:t}, we get
$a(X)\le \psi(X)$.
The inverse inequality $\psi(X)\le a(X)$ follows from Theorem \ref{corollary101:product-selection}.
\end{proof}

\section{Questions}
\label{section:questions}

It is unclear if Theorem \ref{thm:Gutev} remains valid for separately continuous weak selections:

\begin{question}
If the square $X\times X$ of a regular countably compact space $X$ has a separately continuous weak selection, must $X$ be compact and metrizable?
\end{question}

It follows from 
Corollary \ref{ultrafilter:is:weakly:orderable} and Proposition
\ref{proposition:ultra-not-subord} that the first implication 
in \eqref{four:properties:for:filter:spaces} is not reversible for filter spaces. To the best of our knowledge,
it is unknown if the second implication in \eqref{four:properties:for:filter:spaces} is reversible for filter spaces.

\begin{question}
Is every filter space with a continuous weak selection  weakly orderable?
\end{question}

Every weakly $p$-orderable filter space $X_p$ is (trivially) weakly orderable and satisfies the inequality $\psi(X_p)\le a(X_p)$ by Proposition \ref{proposition:p-ord->psi<a}; see also Diagram 1.
One may wonder if these two properties characterize weak $p$-orderability of a filter space $X_p$.

\begin{question}
\label{question6:product-selection}
Is every weakly orderable filter space $X_{p}$ satisfying 
$\psi(X_p)\le a(X_p)$ weakly $p$-orderable?
\end{question}

It is unclear if the converse of Corollary 
\ref{corollary102:product-selection} holds.

\begin{question}
\label{square:of:lots:question}
Let $X$ be a hereditarily paracompact totally disconnected orderable space satisfying $\psi(X) = a(X)$. Does $X \times X$ admit a (separately) continuous weak selection?
\end{question}

\section*{Acknowledgments}

The authors would like to thank the referee for careful
reading of the manuscript and helpful comments.

\end{document}